\documentclass{amsart}

\newtheorem{thm}{Theorem}[section]
\newtheorem{dfn}[thm]{Definition}
\newtheorem{prp}[thm]{Proposition}
\newtheorem{lem}[thm]{Lemma}

\newcommand \F {\mathcal F}
\newcommand \Fix {\operatorname{Fix}}
\renewcommand \Im {\operatorname{Im}}
\newcommand \N {\mathbb N}
\newcommand \Z {\mathbb Z}
\newcommand \abs[1] {\lvert #1 \rvert}
\newcommand \geom[1] {\lVert #1 \rVert}
\newcommand \id {\mathrm{id}}
\newcommand \pt {\mathbf{pt}}
\newcommand \susp {\Sigma}
\newcommand \z {\mathbf z}

\begin{document}

\title{Homotopy type of Frobenius complexes}
\author{TOUNAI Shouta}
\address{Graduate School of Mathematical Sciences, The University of Tokyo, 3-8-1 Komaba, Meguro, Tokyo, 153-8914 Japan}
\email{tounai@ms.u-tokyo.ac.jp}

\begin{abstract}
  A submonoid $\Lambda$ of ${\mathbb N}^d$ has a natural order defined by ${\lambda \le \lambda + \mu}$
  for ${\lambda, \mu \in \Lambda}$.
  The Frobenius complex is the order complex of an open interval of $\Lambda$ with respect to this order.
  In this paper, the homotopy type of the Frobenius complex of $\Lambda$ is determined
  when $\Lambda$ is the submonoid of $\mathbb N$ generated by two relatively prime integers, or
  the submonoid of ${\mathbb N}^2$ generated by three elements of which any two are linearly independent.
  As an application,
  the multigraded Poincar\'e series of the quotient algebra $K[x, y, z] / (x^p y^q - z^r)$ over a field $K$
  is determined and proved to be rational.
\end{abstract}
\maketitle

\section{Introduction}

An additive monoid $\Lambda$ has a natural preorder $\le_\Lambda$ defined by
\[ \lambda \le_\Lambda \mu \iff {^\exists \nu \in \Lambda} \text{ such that } \lambda + \nu = \mu. \]
If $\Lambda$ is cancellative
(i.e. $\lambda + \nu = \mu + \nu$ implies $\lambda = \mu$ for any $\lambda, \mu, \nu \in \Lambda$) and
has no non-trivial inverses,
then $\le_\Lambda$ is a partial order on $\Lambda$.
For such an additive monoid $\Lambda$ and $\lambda \in \Lambda_{+} := \Lambda \setminus \{0\}$,
define the \emph{Frobenius complex} $\F(\lambda; \Lambda)$ of $\Lambda$ for $\lambda$ by
\[ \F(\lambda; \Lambda) = \geom{(0, \lambda)_\Lambda}, \]
where $\geom P$ denotes the geometric realization of a poset $P$,
and $(0, \lambda)_\Lambda$ denotes
the open interval $\{\, \mu \in \Lambda \mid 0 <_\Lambda \mu <_\Lambda \lambda \,\}$ in $\Lambda$
(given the restricted order of $\le_\Lambda$).
The symbol $<_\Lambda$ means the strict order associated to $\le_\Lambda$, i.e.
\[ \lambda <_\Lambda \mu \iff \lambda \le_\Lambda \mu \text{ and } \lambda \neq \mu
  \iff {^\exists \nu \in \Lambda_{+}} \text{ such that } \lambda + \nu = \mu. \]
We will denote $\le_\Lambda$ and $<_\Lambda$ simply by $\le$ and $<$ respectively when no confusion can arise.

Clark and Ehrenborg~\cite{CE} determined the homotopy type of the Frobenius complex of $\Lambda_1$
using discrete Morse theory
when $\Lambda_1$ is the submonoid of $\N$ generated by two relatively prime integers.
According to this theorem, the Frobenius complex of $\Lambda_1$
is either contractible or homotopy equivalent to a sphere.
In this paper, we give a shorter and simpler calculation for the Frobenius complex of $\Lambda_1$,
in which we construct an isomorphism between $\Lambda_1$ and a quotient monoid $\Lambda^{p,q}$
of the free additive monoid $\N^2$,
introduce a connection between $\Lambda^{p,q}$ and a much simpler quotient monoid $\Lambda^{2,2}$,
and determine the homotopy type of the Frobenius complex of $\Lambda^{2,2}$.
Moreover, we also determine the homotopy type of the Frobenius complex of $\Lambda_2$
when $\Lambda_2$ is the submonoid of $\N^2$ generated by three elements of which any two are linearly independent.
This calculation is based on a similar idea.
As a result, we proved that the Frobenius complex of $\Lambda_2$
is either contractible or homotopy equivalent to a sphere or a wedge of two spheres of the same dimension.
The details are stated in Proposition~\ref{prp : 3 gen quotient}, Theorem~\ref{thm : p,q,r to 2,2,2},
and Theorem~\ref{thm : case 2,2,2}.

Laudal and Sletsj{\o}e~\cite{LS} and Peeva, Reiner and Sturmfels~\cite{PRS} associated
the Betti numbers of the Frobenius complex of $\Lambda$
with the multigraded Poincar\'e series of the monoid algebra of $\Lambda$.
Using this connection and the result of this paper,
we determine the multigraded Poincar\'e series of the quotient algebra ${K[x, y, z] / (x^p y^q - z^r)}$
over a field $K$ for positive integers $p, q$ and $r$,
and prove that this formal power series is given by a rational function.

This paper is organized in the following way.
In Section~\ref{sec : two gen} we introduce some notions and
calculate the homotopy type of the Frobenius complex of $\Lambda_1$.
In Section~\ref{sec : three gen} we proceed to calculation
for the homotopy type of the Frobenius complex of $\Lambda_2$.
Section~\ref{sec : application} is devoted to an application.
We calculate the graded Poincar\'e series of the monoid algebra
using the result of the previous section.

\section{Submonoids of $\N$ with two generator}
\label{sec : two gen}

In this paper, $\N$ denotes the additive monoid $(\{0, 1, 2, \dotsc\}, {+})$ of non-negative integers.
First, we consider the submonoid $\N p + \N q$ of $\N$ generated by two relatively prime integers $p$ and $q$.
For positive integers $p$ and $q$, define
\[ \Lambda^{p,q} = \langle\, a, b \mid p a = q b \,\rangle, \]
where the right hand side means the quotient monoid of the free additive monoid $\N a \oplus \N b$
by the equivalence relation $\sim$ generated by
\[ \lambda + p a \sim \lambda + q b \qquad (\lambda \in \N a \oplus \N b). \]

\begin{prp}
  Let $p$ and $q$ be relatively prime integers.
  Then $\N p + \N q$ is isomorphic to $\Lambda^{p,q}$.
\end{prp}
\begin{proof}
  Consider the monoid homomorphism $f : \Lambda^{p,q} \to \N$ which sends $a$ and $b$ to $q$ and $p$ respectively.
  Then $f$ is well-defined and injective.
  Hence $f$ induces a monoid isomorphism $\Lambda^{p,q} \cong \Im f = \N p + \N q$.
\end{proof}

When either $p = 1$ or $q = 1$ holds, the submonoid $\N p + \N q$ equals to $\N$.
In this case, the Frobenius complex are easily determined.

\begin{prp}
  The Frobenius complex of $\N$ for $n \in \N_{+}$ satisfies
  \[ \F(n; \N) \simeq
    \begin{cases}
      S^{-1} & \text{if $n = 1$,} \\
      \pt & \text{if $n \ge 2$,}
    \end{cases} \]
  where $S^n$ denotes the $n$-sphere
  (in particular, $S^{-1}$ means the empty space),
  and $\pt$ the one-point space.
\end{prp}

We now turn to the case $p, q \ge 2$.
We will see that we need only consider $\Lambda^{2,2}$.
For positive integers $p$ and $q$, define the \emph{transition function} $\tau^p_q : \N \to \N$ by
\[ \tau^p_q(m p + n) = m q + \min\{n, q-1\} \qquad (m \in \N, \ n \in \N_{< p}). \]
The following properties are easy to check.

\begin{prp}
  \label{prp : trans func}
  The transition functions satisfy the following:
  \begin{enumerate}
    \item $\tau^p_q(n + p) = \tau^p_q(n) + q$ for each $n \in \N$,
    \item $\tau^p_q$ is order-preserving,
    \item $\tau^q_p \circ \tau^p_q \le \id_\N$,
    \item $\tau^q_p \circ \tau^p_q = \id_\N$ if $p \le q$, and
    \item $(\tau^p_q)^{-1}(0) = \{0\}$ if either $p = q = 1$ or $p, q \ge 2$ holds.
  \end{enumerate}
\end{prp}

For positive integers $p, q, r$ and $s$,
define the \emph{transition map} $T^{p,q}_{r,s} : \Lambda^{p,q} \to \Lambda^{r,s}$ by
\[ T^{p,q}_{r,s}(m a + n b) = \tau^p_r(m) a + \tau^q_s(n) b \qquad (m, n \in \N). \]
Proposition~\ref{prp : trans func} implies the following properties.

\begin{prp}
  \label{prp : trans map}
  The transition maps satisfy the following:
  \begin{enumerate}
    \item $T^{p,q}_{r,s}$ is well-defined,
    \item $T^{p,q}_{r,s}$ is order-preserving,
    \item $T^{r,s}_{p,q} \circ T^{p,q}_{r,s} \le \id_{\Lambda^{p,q}}$,
    \item $T^{r,s}_{p,q} \circ T^{p,q}_{r,s} = \id_{\Lambda^{p,q}}$
      if $p \le r$ and $q \le s$, and
    \item $(T^{p,q}_{r,s})^{-1}(0) = \{0\}$
      if $p, q, r, s \ge 2$.
  \end{enumerate}
\end{prp}

The following is a key theorem of this paper.

\begin{thm}
  \label{thm : p,q to 2,2}
  For $p, q \in \N_{\ge 2}$
  the Frobenius complex of $\Lambda^{p,q}$ for $\lambda \in \Lambda^{p,q}_{+}$ satisfies
  \[ \F(\lambda; \Lambda^{p,q}) \simeq
    \begin{cases}
      \F(T^{p,q}_{2,2}(\lambda); \Lambda^{2,2}) & \text{if $\lambda \in \Fix(T^{2,2}_{p,q} \circ T^{p,q}_{2,2})$,} \\
      \pt & \text{otherwise,}
    \end{cases} \]
  where $\Fix f$ denotes the set of all fixed points of $f$, i.e.
  \[ \Fix f = \{\, x \in X \mid f(x) = x \,\} \]
  for a function $f$ on a set $X$.
\end{thm}

We introduce the following notion to prove the theorem.

\begin{dfn}
  Let $P$ be a poset.
  A downward closure operator on $P$
  is an order-preserving map $f : P \to P$ which satisfies both $f \le \id_P$ and $f^2 = f$.
\end{dfn}

\begin{lem}
  \label{lem : downward}
  Let $f$ be a downward closure operator on a poset $P$, and
  $a, b \in P$ with $a < b$.
  Assume that $f^{-1}(a) = \{a\}$.
  Then we have
  \[ \geom{(a, b)_P} \simeq
    \begin{cases}
      \geom{(a, b)_{\Fix f}} & \text{if $b \in \Fix f$,} \\
      \pt & \text{otherwise.}
    \end{cases} \]
\end{lem}
\begin{proof}
  We first show that $f((a, b)_P) \subset (a, b)_P \cap \Fix f$.
  Let $x \in (a, b)_P$.
  Since $f^{-1}(a) = \{a\}$, we have $a < f(x)$.
  Since $f$ is a downward closure operator,
  we have $f(x) \le x < b$ and $f(f(x)) = f(x)$.

  Hence $f$ induces an order-preserving map $(a, b)_P \to (a, b)_P \cap \Fix f$, say $r$.
  Let $i$ be the inclusion $(a, b)_P \cap \Fix f \hookrightarrow (a, b)_P$.
  The geometric realizations of $i$ and $r$ are mutually homotopy inverse,
  since $r i = \id$ and $i r \le \id$.
  Thus $\geom{(a, b)_P}$ is homotopy equivalent to $\geom{(a, b)_P \cap \Fix f}$.
  When $b \in \Fix f$, the proof is complete since $(a, b)_P \cap \Fix f = (a, b)_{\Fix f}$.
  When $b \notin \Fix f$,
  it is easy to check that $f(b)$ is the maximum element of $(a, b)_P \cap \Fix f$.
  Thus we have $\geom{(a, b)_P \cap \Fix f} \simeq \pt$.
\end{proof}

\begin{proof}
  [Proof of Theorem~\ref{thm : p,q to 2,2}]
  By Proposition~\ref{prp : trans map},
  $T^{2,2}_{p,q} \circ T^{p,q}_{2,2}$ is a downward closure operator on $\Lambda^{p,q}$,
  and satisfies $(T^{2,2}_{p, q} \circ T^{p,q}_{2,2})^{-1}(0) = \{0\}$.
  By Lemma~\ref{lem : downward}, we obtain
  \[ \F(\lambda; \Lambda^{p,q}) = \geom{(0, \lambda)_{\Lambda^{p,q}}} \simeq
    \begin{cases}
      \geom{(0, \lambda)_{\Fix(T^{2,2}_{p,q} \circ T^{p,q}_{2,2})}}
      & \text{if $\lambda \in \Fix(T^{2,2}_{p,q} \circ T^{p,q}_{2,2})$,} \\
      \pt & \text{otherwise.}
    \end{cases} \]
  Assume that $\lambda \in \Fix(T^{2,2}_{p,q} \circ T^{p,q}_{2,2})$.
  Since $T^{p,q}_{2,2}$ and $T^{2,2}_{p,q}$ induce a poset isomorphism
  $\Fix(T^{2,2}_{p,q} \circ T^{p,q}_{2,2}) \cong \Lambda^{2,2}$,
  we have
  \[ \geom{(0, \lambda)_{\Fix(T^{2,2}_{p,q} \circ T^{p,q}_{2,2})}}
    \cong \geom{(T^{p,q}_{2,2}(0), T^{p,q}_{2,2}(\lambda))_{\Lambda^{2,2}}}
    = \F(T^{p,q}_{2,2}; \Lambda^{2,2}). \qedhere \]
\end{proof}

We now consider the Frobenius complexes of $\Lambda^{2,2}$.

\begin{thm}
  \label{thm : case 2,2}
  The Frobenius complex $\F(\lambda; \Lambda^{2,2})$ of $\Lambda^{2,2}$
  for $\lambda = m a + n b \in \Lambda^{2,2}_{+}$
  is homotopy equivalent to $S^{m + n - 2}$.
\end{thm}

To prove the theorem, we use the following lemma.

\begin{lem}
  \label{lem : susp}
  Let $X_1$ and $X_2$ be subcomplexes of a CW complex of $X$.
  If both $X_1$ and $X_2$ are contractible, then we have
  \[ X_1 \cup X_2 \simeq \susp(X_1 \cap X_2), \]
  where $\susp Y$ denotes the (unreduced) suspension of a topological space $Y$.
\end{lem}

\begin{proof}[Proof of Theorem~\ref{thm : case 2,2}]
  The proof is by induction on $m + n$.
  If $m + n = 1$, i.e. $\lambda = a$ or $\lambda = b$, we have
  \[ \F(\lambda; \Lambda^{2,2}) = \geom{(0, \lambda)_{\Lambda^{2,2}}}
    = \geom \emptyset = S^{-1}. \]
  Let $m + n \ge 2$, and assume that $m \ge 1$.
  Since $\Lambda^{2,2}$ is generated by $a$ and $b$, we have
  \[ \mu < \nu \iff \mu + a \le \nu \text{ or } \mu + b \le \nu \]
  for any $\mu, \nu \in \Lambda^{2,2}$.
  Thus we obtain $(0, \lambda)_{\Lambda^{2,2}} = [a, \lambda)_{\Lambda^{2,2}} \cup [b, \lambda)_{\Lambda^{2, 2}}$.
  Moreover, since any chain in $(0, \lambda)_{\Lambda^{2,2}}$ is contained
  in either $[a, \lambda)_{\Lambda^{2,2}}$ or $[b, \lambda)_{\Lambda^{2,2}}$, we have
  \[ \F(\lambda, \Lambda^{2,2}) = \geom{(0, \lambda)_{\Lambda^{2,2}}}
  = \geom{[a, \lambda)_{\Lambda^{2,2}}} \cup \geom{[b, \lambda)_{\Lambda^{2, 2}}}. \]
  By Lemma~\ref{lem : susp}, we obtain
  \[ \geom{[a, \lambda)_{\Lambda^{2,2}}} \cup \geom{[b, \lambda)_{\Lambda^{2, 2}}}
    \simeq \susp(\geom{[a, \lambda)_{\Lambda^{2,2}}} \cap \geom{[b, \lambda)_{\Lambda^{2, 2}}}), \]
  since $a$ and $b$ are the minimal element of $[a, \lambda)_{\Lambda^{2,2}}$ and $[b, \lambda)_{\Lambda^{2,2}}$ respectively.
  It is easy to check that
  \[ a \le \mu \text{ and } b \le \mu \iff a + a \le \mu \text{ or } a + b \le \mu \]
  for any $\mu \in \Lambda^{2,2}$. Thus we have
  \begin{align*}
    \geom{[a, \lambda)_{\Lambda^{2,2}}} \cap \geom{[b, \lambda)_{\Lambda^{2, 2}}}
      &= \geom{[a, \lambda)_{\Lambda^{2,2}} \cap [b, \lambda)_{\Lambda^{2, 2}}}
      \\&= \geom{[a + a, \lambda)_{\Lambda^{2,2}} \cup [a + b, \lambda)_{\Lambda^{2, 2}}}
      \\&= \geom{(a, \lambda)_{\Lambda^{2,2}}}
      \\&\cong \geom{(0, \lambda - a)_{\Lambda^{2,2}}}.
      \\&= \F(\lambda - a; \Lambda^{2,2})
  \end{align*}
  Note that the map on $\Lambda^{2,2}$ which sends $\mu$ to $\mu + a$ induces
  a poset isomorphism $(0, \lambda - a)_{\Lambda^{2,2}} \cong (a, \lambda)_{\Lambda^{2,2}}$.
  By the inductive assumption, we conclude
  \[ \F(m a + n b; \Lambda^{2,2}) \simeq \susp \F((m - 1) a + n b; \Lambda^{2,2}) \simeq \susp S^{(m - 1) + n - 2}
    \cong S^{m + n - 2}. \]

  If $m = 0$, then we also have
  \[ \F(n b; \Lambda^{2,2}) = \F(2 a + (n - 2) b; \Lambda^{2,2}) \simeq S^{2 + (n-2) - 2} = S^{n-2}. \qedhere \]
\end{proof}

\section{Submonoids of $\N^2$ with three generators}
\label{sec : three gen}

For positive integers $p, q$ and $r$, define
\[ \Lambda^{p,q,r} = \langle\, a, b, c \mid p a + q b = r c \,\rangle, \]
where the right hand side means the quotient monoid of the free additive monoid $\N a \oplus \N b \oplus \N c$
by the equivalence relation $\sim$ generated by
\[ \lambda + p a + q b \sim \lambda + r c \qquad (\lambda \in \N a \oplus \N b \oplus \N c). \]

\begin{prp}
  \label{prp : 3 gen quotient}
  Let $u, v, w \in \N^2$ and any two of $u, v, w$ are linearly independent.
  Then the submonoid $\N u + \N v + \N w$ of $\N^2$ generated by $u, v$ and $w$
  is isomorphic to $\Lambda^{p,q,r}$ for some positive integers $p, q$ and $r$.
\end{prp}
\begin{proof}
  Take the generator $(m_0, n_0, k_0)$ of $\{\, (m, n, k) \in \Z^3 \mid m u + n v + k w = 0 \,\}$
  as $\Z$-module.
  Since any two of $u, v, w$ are linearly independent, each of $m_0, n_0, k_0$ not equals to zero.
  Moreover, since $u, v, w \in \N^2$, one of $m_0, n_0, k_0$ has a different signature from the others.
  By permutation on $u, v, w$ and choice of the generator,
  we can assume that $m_0, n_0 > 0$ and $k_0 < 0$.
  Set $p = m_0$, $q = n_0$, $r = -k_0$, and
  consider the monoid homomorphism $f : \Lambda^{p,q,r} \to \N^2$
  which sends $a, b$ and $c$ to $u, v$ and $w$ respectively.
  Then $f$ is well-defined and injective.
  Hence $f$ induces a monoid isomorphism $\Lambda^{p,q,r} \cong \Im f = \N u + \N v + \N w$.
\end{proof}

By the previous proposition, it is enough to consider the case $\Lambda = \Lambda^{p,q,r}$
for positive integers $p, q$ and $r$.
When $r = 1$, obviously we have $\Lambda^{p,q,1} \cong \N a \oplus \N b$.
In this case, the Frobenius complex is easily determined.
\begin{prp}
  \label{prp : p,q,1}
  The Frobenius complex of $\N a \oplus \N b$ for $m a + n b \in (\N a \oplus \N b)_{+}$ satisfies
  \[ \F(m a + n b; \N a \oplus \N b) \simeq
    \begin{cases}
      S^{-1} & \text{if $(m, n) = (1, 0)$ or $(0, 1)$,} \\
      S^0 & \text{if $(m, n) = (1, 1)$,} \\
      \pt & \text{if either $m \ge 2$ or $n \ge 2$.}
    \end{cases} \]
\end{prp}
\begin{proof}
  Consider the map $f : \N a \oplus \N b \to \N a \oplus \N b$ defined by
  \[ f(m a + n b) = \min\{m, 1\} a + \min\{n, 1\} b \qquad (m, n \in \N). \]
  Then $f$ is a downward closure operator and satisfies $f^{-1}(0) = \{0\}$ and
  \[ \Fix f = \{ 0, a, b, a + b \}. \]
  By Lemma~\ref{lem : downward}, we obtain the desired conclusion.
\end{proof}

We now turn to the case $r \ge 2$.
For positive integers $p, q, r, s, t$ and $u$,
define the transition map $T^{p,q,r}_{s,t,u} : \Lambda^{p,q,r} \to \Lambda^{s,t,u}$ by
\[ T^{p,q,r}_{s,t,u}(m a + n b + k c) = \tau^p_s(m) a + \tau^q_t(n) b + \tau^r_u(k) c
  \qquad (m, n, k \in \N). \]
Proposition~\ref{prp : trans func} implies the following properties.

\begin{prp}
  The transition maps satisfy the following:
  \begin{enumerate}
    \item $T^{p,q,r}_{s,t,u}$ is well-defined,
    \item $T^{p,q,r}_{s,t,u}$ is order-preserving,
    \item $T^{s,t,u}_{p,q,r} \circ T^{p,q,r}_{s,t,u} \le \id_{\Lambda^{p,q,r}}$,
    \item $T^{s,t,u}_{p,q,r} \circ T^{p,q,r}_{s,t,u} = \id_{\Lambda^{p,q,r}}$
      if $p \le s$, $q \le t$ and $r \le u$, and
    \item $(T^{p,q,r}_{s,t,u})^{-1}(0) = \{0\}$
      if $(p, s), (q, t), (r, u) \in \{1\}^2 \cup (\N_{\ge 2})^2$.
  \end{enumerate}
\end{prp}

It is enough to consider the case $p \le q$
since $\Lambda^{p,q,r}$ is isomorphic to $\Lambda^{q,p,r}$
by the map $\Lambda^{p,q,r} \to \Lambda^{q,p,r}$ which sends $a, b$ and $c$ to $b, a$ and $c$ respectively.

\begin{thm}
  \label{thm : p,q,r to 2,2,2}
  Let $p, q$ and $r$ be positive integers with $p \le q$ and $r \ge 2$.
  Then the Frobenius complex of $\Lambda^{p,q,r}$ for $\lambda \in \Lambda^{p,q,r}_{+}$ satisfies
  the following:
  \begin{enumerate}
    \item If $p = q = 1$,
      \[ \F(\lambda; \Lambda^{1,1,r}) \simeq
        \begin{cases}
          \F(T^{1,1,r}_{1,1,2}(\lambda); \Lambda^{1,1,2})
          & \text{if $\lambda \in \Fix(T^{1,1,2}_{1,1,r} \circ T^{1,1,r}_{1,1,2})$,} \\
          \pt & \text{otherwise}.
        \end{cases} \]
    \item If $p = 1$ and $q \ge 2$,
      \[ \F(\lambda; \Lambda^{1,q,r}) \simeq
        \begin{cases}
          \F(T^{1,q,r}_{1,2,2}(\lambda); \Lambda^{1,2,2})
          & \text{if $\lambda \in \Fix(T^{1,2,2}_{1,q,r} \circ T^{1,q,r}_{1,2,2})$,} \\
          \pt & \text{otherwise}.
        \end{cases} \]
    \item If $2 \le p \le q$,
      \[ \F(\lambda; \Lambda^{p,q,r}) \simeq
        \begin{cases}
          \F(T^{p,q,r}_{2,2,2}(\lambda); \Lambda^{2,2,2})
          & \text{if $\lambda \in \Fix(T^{2,2,2}_{p,q,r} \circ T^{p,q,r}_{2,2,2})$,} \\
          \pt & \text{otherwise}.
        \end{cases} \]
  \end{enumerate}
\end{thm}
\begin{proof}
  Let $s = \min\{2, p\}$ and $t = \min\{2, q\}$.
  By Proposition~\ref{prp : trans func},
  $T^{s,t,2}_{p,q,r} \circ T^{p,q,r}_{s,t,2}$ is a downward closure operator on $\Lambda^{p,q,r}$
  and satisfies $(T^{s,t,2}_{p,q,r} \circ T^{p,q,r}_{s,t,2})^{-1}(0) = \{0\}$.
  By Lemma~\ref{lem : downward}, we obtain
  \[ \F(\lambda; \Lambda^{p,q,r}) = \geom{(0, \lambda)_{\Lambda^{p,q,r}}} \simeq
    \begin{cases}
      \geom{(0, \lambda)_{\Fix (T^{s,t,2}_{p,q,r} \circ T^{p,q,r}_{s,t,2})}}
      & \text{if $\lambda \in \Fix (T^{s,t,2}_{p,q,r} \circ T^{p,q,r}_{s,t,2})$,} \\
      \pt & \text{otherwise.}
    \end{cases} \]
  Assume that $\lambda \in \Fix (T^{s,t,2}_{p,q,r} \circ T^{p,q,r}_{s,t,2})$.
  Since $T^{p,q,r}_{s,t,2}$ and $T^{s,t,2}_{p,q,r}$ induce a poset isomorphism
  $\Fix (T^{s,t,2}_{p,q,r} \circ T^{p,q,r}_{s,t,2}) \cong \Lambda^{s,t,2}$, we have
  \[ \geom{(0, \lambda)_{\Fix (T^{s,t,2}_{p,q,r} \circ T^{p,q,r}_{s,t,2})}}
    \simeq \geom{(T^{p,q,r}_{s,t,2}(0), T^{p,q,r}_{s,t,2}(\lambda))_{\Lambda^{s,t,2}}}
    \simeq \F(T^{p,q,r}_{s,t,2}(\lambda); \Lambda^{s,t,2}). \qedhere \]
\end{proof}

Note that for any element $\lambda$ of $\Lambda^{p,q,r}$
there uniquely exist non-negative integers $m, n$ and $k$
satisfying $\lambda = m a + n b + k c$ and either $m \le p - 1$ or $n \le q - 1$.

\begin{thm}
  \label{thm : case 2,2,2}
  The Frobenius complexes of $\Lambda^{1,1,2}, \Lambda^{1,2,2}$ and $\Lambda^{2,2,2}$ satisfy the following:
  \begin{enumerate}
    \item For $m a + n b + k c \in \Lambda^{1,1,2}_{+}$ with $m = 0$ or $n = 0$,
      \[ \F(m a + n b + k c; \Lambda^{1,1,2}) \simeq
        \begin{cases}
          S^{-1} & \text{if $m = n = 0$ and $k = 1$,} \\
          S^{k - 2} \vee S^{k - 2} & \text{if $m = n = 0$ and $k \ge 2$,} \\
          S^{k - 1} & \text{if $(m, n) = (1, 0)$ or $(0, 1)$, and} \\
          \pt & \text{otherwise.}
        \end{cases} \]
    \item For $m a + n b + k c \in \Lambda^{1,2,2}_{+}$ with $m = 0$ or $n \le 1$,
      \[ \F(m a + n b + k c; \Lambda^{1,2,2}) \simeq
        \begin{cases}
          S^{m + n + k - 2} & \text{if $m \le 1$ and $n \le 1$, and} \\
          \pt & \text{otherwise.}
        \end{cases} \]
    \item For $m a + n b + k c \in \Lambda^{2,2,2}_{+}$ with $m \le 1$ or $n \le 1$,
      \[ \F(m a + n b + k c; \Lambda^{2,2,2}) \simeq
        \begin{cases}
          S^{m + n + k - 2} & \text{if $m \le 1$ and $n \le 1$, and} \\
          \pt & \text{otherwise.}
        \end{cases} \]
  \end{enumerate}
\end{thm}
\begin{proof}
  The proof is by induction on $k$.
  We first show the inductive step.
  Let $(p, q, r) = (1, 1, 2), (1, 2, 2)$ or $(2, 2, 2)$, and $\lambda \in \Lambda^{p,q,r}$
  with $c < \lambda$ and $a + b < \lambda$.
  Since $\Lambda^{p,q,r}$ is generated by $a, b$ and $c$,
  we have
  \begin{align*}
    \F(\lambda; \Lambda^{p,q,r})
    &= \geom{(0, \lambda)_{\Lambda^{p,q,r}}}
    \\&= \geom{[a, \lambda)_{\Lambda^{p,q,r}}}
    \cup \geom{[b, \lambda)_{\Lambda^{p,q,r}}}
    \cup \geom{[c, \lambda)_{\Lambda^{p,q,r}}}.
  \end{align*}
  Since $\{a, b\}$ has the supremum $a + b$ in $\Lambda^{p,q,r}$, we have
  \begin{align*}
    \geom{[a, \lambda)_{\Lambda^{p,q,r}}} \cup \geom{[b, \lambda)_{\Lambda^{p,q,r}}}
    &\simeq \susp(\geom{[a, \lambda)_{\Lambda^{p,q,r}}} \cap \geom{[b, \lambda)_{\Lambda^{p,q,r}}})
    \\&\simeq \susp \geom{[a + b, \lambda)_{\Lambda^{p,q,r}}}
    \\&\simeq \susp \pt
    \\&\simeq \pt.
  \end{align*}
  It is easy to check that
  \begin{align*}
    (a \le \mu \text{ or } b \le \mu) \text{ and } c \le \mu
      &\iff a + c \le \mu \text{ or } b + c \le \mu \text{ or } c + c \le \mu
      \\&\iff c < \mu
    \end{align*}
  for $\mu \in \Lambda^{p,q,r}$.
  Thus we have
  \begin{align*}
    \F(\lambda; \Lambda^{p,q,r})
    &= \Bigl( \geom{[a, \lambda)_{\Lambda^{p,q,r}}}
      \cup \geom{[b, \lambda)_{\Lambda^{p,q,r}}} \Bigr)
      \cup \geom{[c, \lambda)_{\Lambda^{p,q,r}}}
    \\&\simeq \susp \Bigl[ \Bigl( \geom{[a, \lambda)_{\Lambda^{p,q,r}}} \cup \geom{[b, \lambda)_{\Lambda^{p,q,r}}} \Bigr)
      \cap \geom{[c, \lambda)_{\Lambda^{p,q,r}}} \Bigr]
    \\&= \susp \geom{(c, \lambda)_{\Lambda^{p,q,r}}}
    \\&\cong \susp \geom{(0, \lambda - c)_{\Lambda^{p,q,r}}}
    \\&= \susp \F(\lambda - c; \Lambda^{p,q,r}).
  \end{align*}

  We now turn to the case $(p, q, r) = (1, 1, 2)$.
  Define $f : \Lambda^{1,1,2} \to \Lambda^{1,1,2}$ by
  \[ f(m a + n b + k c) = \min\{m, n + 1\} a + \min\{m + 1, n\} b + k c \qquad(m, n, k \in \N). \]
  Then $f$ is well-defined and a downward closure operator on $\Lambda^{1,1,2}$,
  and satisfies $f^{-1}(0) = \{0\}$ and
  \[ \Fix f = \{\, m a + n b + k c \in \Lambda^{1,1,2} \mid \abs{m - n} \le 1 \,\}. \]
  By Lemma~\ref{lem : downward}, $\F(m a + n b + k c; \Lambda^{1,2,2}) \simeq \pt$
  if either $m = 0$ and $n \ge 2$, or $m \ge 2$ and $n = 0$ hold.
  We need only consider three cases: $\lambda = k c$, $\lambda = a + k c$, and $\lambda = b + k c$.
  Clearly, we have $\F(c, \Lambda^{1,1,2}) = \geom{\emptyset} = S^{-1}$ and
  $\F(2 c, \Lambda^{1,1,2}) = \geom{\{a, b, c\}} \cong S^0 \vee S^0$.
  When $k \ge 3$, we have
  \[ \F(k c; \Lambda^{1,1,2}) \simeq \susp \F((k - 1) c; \Lambda^{1,1,2}) \simeq
    \susp(S^{(k - 1) - 2} \vee S^{(k - 1) - 2}) \simeq S^{k - 2} \vee S^{k - 2} \]
  by induction on $k$.
  Similarly we have $\F(a; \Lambda^{1,1,2}) = S^{-1}$, $\F(a + c; \Lambda^{1,1,2}) \cong S^0$ and
  \[ \F(a + k c; \Lambda^{1,1,2}) \simeq \susp \F(a + (k - 1)c; \Lambda^{1,1,2})
    \simeq \susp S^{(k - 1) - 1} \simeq S^{k - 1} \qquad ( k \ge 2). \]
  By the isomorphism on $\Lambda^{1,1,2}$ swapping $a$ and $b$, we have
  \[ \F(b + k c; \Lambda^{1,1,2}) \cong \F(a + k c; \Lambda^{1,1,2}) \simeq S^{k - 1}. \]

  We now turn to the case $(p, q, r) = (1, 2, 2)$.
  Define $g : \Lambda^{1,2,2} \to \Lambda^{1,2,2}$ by
  \begin{align*}
    g(m a + n b + k c)
    &= \min\{m, \tau^2_1(n) + 1\} a + \min\{\tau^1_2(m) + 1, n\} b + k c
    \\&= \min\{m, \lfloor \tfrac n 2 \rfloor + 1\} a + \min\{2 m + 1, n\} b + k c
    \qquad(m, n, k \in \N).
  \end{align*}
  Then $g$ is well-defined and a downward closure operator on $\Lambda^{1,2,2}$,
  and satisfies $g^{-1}(0) = \{0\}$ and
  \[ \Fix g = \{\, m a + n b + k c \in \Lambda^{1,2,2} \mid -1 \le 2 m - n \le 2 \,\}. \]
  Thus we obtain $\F(m a + n b + k c; \Lambda^{1,2,2}) \simeq \pt$
  if either $m = 0$ and $n \ge 2$, or $m \ge 2$ and $n \le 1$ hold.
  It is easily seen that $\F(c; \Lambda^{1,2,2}) = \geom{\emptyset} = S^{-1}$, and we have
  \[ \F(k c; \Lambda^{1,2,2}) \simeq \susp \F((k - 1)c; \Lambda^{1,1,2})
    \simeq \susp S^{(k - 1) - 1} \simeq S^{k - 1} \qquad (k \ge 2) \]
  by induction on $k$.
  By similar arguments, we have
  \begin{align*}
    \F(a + k c; \Lambda^{1,2,2}) &\simeq S^{k - 1}, \\
    \F(b + k c; \Lambda^{1,2,2}) &\simeq S^{k - 1}, \text{and}\\
    \F(a + b + k c; \Lambda^{1,2,2}) &\simeq S^k \qquad (k \ge 0).
  \end{align*}

  We now turn to the last case $(p, q, r) = (2, 2, 2)$.
  Define $h : \Lambda^{2,2,2} \to \Lambda^{2,2,2}$ by
  \[ h(m a + n b + k c) = \min\{m, n + 1\} a + \min\{m + 1, n\} b + k c \qquad(m, n, k \in \N). \]
  Then $h$ is well-defined and a downward closure operator on $\Lambda^{2,2,2}$,
  and satisfies $h^{-1}(0) = \{0\}$ and
  \[ \Fix h = \{\, m a + n b + k c \in \Lambda^{2,2,2} \mid \abs{m - n} \le 1 \,\}. \]
  Thus we have $\F(m a + n b + k c; \Lambda^{2,2,2}) \simeq \pt$
  if $\abs{m - n} \ge 2$.
  It is easily seen that $\F(c; \Lambda^{2,2,2}) = \geom{\emptyset} = S^{-1}$, and we have
  \[ \F(k c; \Lambda^{2,2,2}) \simeq \susp \F((k - 1) c; \Lambda^{2,2,2}) \simeq \susp S^{(k - 1) - 2} = S^{k - 2}
    \qquad (k \ge 2) \]
  by induction on $k$.
  By similar arguments, we have
  \begin{align*}
    \F(a + k c; \Lambda^{2,2,2}) &\simeq S^{k - 1}, \\
    \F(b + k c; \Lambda^{2,2,2}) &\simeq S^{k - 1}, \\
    \F(a + b + k c; \Lambda^{2,2,2}) &\simeq S^k, \\
    \F(2 a + b + k c; \Lambda^{2,2,2}) &\simeq \pt, \text{and}\\
    \F(a + 2 b + k c; \Lambda^{2,2,2}) &\simeq \pt \qquad (k \ge 0). \qedhere
  \end{align*}
\end{proof}

\section{An application}
\label{sec : application}

Let $\Lambda$ be an additive submonoid with cancellation law and no non-trivial inverses,
and fix a field $K$.
Define the \emph{$i$-th local Betti number} $\beta_i(\lambda; \Lambda)$ of $\Lambda$ for $\lambda \in \Lambda$ by
\[ \beta_i(\lambda; \Lambda) =
  \begin{cases}
    1 & \text{if $\lambda = 0$ and $i = 0$,} \\
    0 & \text{if $\lambda = 0$ and $i \ge 1$,} \\
    0 & \text{if $\lambda \in \Lambda_{+}$ and $i = 0$,} \\
    \dim_K \widetilde H_{i-2}(\F(\lambda; \Lambda); K) & \text{if $\lambda \in \Lambda_{+}$ and $i \ge 1$.}
  \end{cases} \]
Note that for a topological space $X$
the $(-1)$-th reduced homology $\widetilde H_{-1}(X; K)$ equals to $K$ if $X$ is empty, and $0$ otherwise.
The \emph{monoid algebra} $K[\Lambda]$ is the $K$-algebra with underlying space
\[ K[\Lambda] = \bigoplus_{\lambda \in \Lambda} K \z^\lambda \]
and multiplication
\[ \z^\lambda \cdot \z^\mu = \z^{\lambda + \mu} \qquad (\lambda, \mu \in \Lambda). \]
The following theorem is known (see \cite{LS} and \cite{PRS}).

\begin{thm}
  [Laudal-Sletsj{\o}e \cite{LS}]
  The multigraded Poincar\'e series $P^{K[\Lambda]}_K(t, \z)$ of the monoid algebra $K[\Lambda]$ over $K$
  satisfies
  \[ P^{K[\Lambda]}_K(t, \z) = \sum_{i \ge 0} \sum_{\lambda \in \Lambda} \beta_i(\lambda; \Lambda) t^i \z^\lambda. \]
\end{thm}

Consider the case $\Lambda = \Lambda^{p,q,r}$ for positive integers $p, q$ and $r$.
We remark that the local Betti number of $\Lambda^{p,q,r}$ does not depend on the field $K$
since the Frobenius complex of $\Lambda^{p,q,r}$ is contractible or homotopy equivalent to a sphere
or a wedge of spheres.

\begin{prp}
  The monoid algebra $K[\Lambda^{p,q,r}]$ of $\Lambda^{p,q,r}$ is isomorphic to
  the quotient algebra $K[x, y, z] / (x^p y^q - z^r)$ of the polynomial ring.
\end{prp}
\begin{proof}
  Consider the ring homomorphism $\varphi : K[x, y, z] \to K[\Lambda^{p,q,r}]$ which sends
  $x, y$ and $z$ to $\z^a, \z^b$ and $\z^c$ respectively.
  Then $\varphi$ is surjective and its kernel is generated by $x^p y^q - z^r$.
\end{proof}

Proposition~\ref{prp : p,q,1} gives the following formula:

\begin{prp}
  The multigraded Poincar\'e series of $K[\Lambda^{p,q,1}]$ satisfies
  \[ P^{K[\Lambda^{p,q,1}]}_K(t, \z) = 1 + t \z^a + t \z^b + t^2 \z^{a + b} = (1 + t \z^a)(1 + t \z^b). \]
\end{prp}

We now turn to the case $r \ge 2$.

\begin{prp}
  \label{prp : P p,q,r}
  Let $p, q$ and $r$ be positive integers with $p \le q$ and $r \ge 2$.
  Then the multigraded Poincar\'e series of $K[\Lambda^{p,q,r}]$ satisfies the following:
  \begin{enumerate}
    \item If $p = q = 1$,
      \[ P^{K[\Lambda^{1,1,r}]}_K(t, \z) = (T^{1,1,2}_{1,1,r})_{*} (P^{K[\Lambda^{1,1,2}]}_K(t, \z)). \]
      Here $(T^{1,1,2}_{1,1,r})_{*}$ denotes the linear map which sends $t^i \z^\mu$ to $t^i \z^{T^{1,1,2}_{1,1,r}(\mu)}$.
    \item If $p = 1$ and $q \ge 2$,
      \[ P^{K[\Lambda^{1,q,r}]}_K(t, \z) = (T^{1,2,2}_{1,q,r})_{*} (P^{K[\Lambda^{1,2,2}]}_K(t, \z)). \]
    \item If $p, q \ge 2$,
      \[ P^{K[\Lambda^{p,q,r}]}_K(t, \z) = (T^{2,2,2}_{p,q,r})_{*} (P^{K[\Lambda^{2,2,2}]}_K(t, \z)). \]
  \end{enumerate}
\end{prp}
\begin{proof}
  We prove the case $p, q \ge 2$;
  the other cases can be proved in the same way.
  By Theorem~\ref{thm : p,q,r to 2,2,2}, we have
  \[ \beta_i(\lambda; \Lambda^{p,q,r}) =
    \begin{cases}
      \beta_i(T^{p,q,r}_{2,2,2}(\lambda); \Lambda^{2,2,2})
      & \text{if $\lambda \in \Fix(T^{2,2,2}_{p,q,r} \circ T^{p,q,r}_{2,2,2})$,} \\
      0 & \text{otherwise.}
    \end{cases} \]
  Thus we obtain
  \begin{align*}
    P^{K[\Lambda^{p,q,r}]}_K(t, \z)
    &= \sum_{i \ge 0} \sum_{\lambda \in \Lambda^{p,q,r}} \beta_i(\lambda; \Lambda^{p,q,r}) t^i \z^\lambda
    \\&= \sum_{i \ge 0} \sum_{\lambda \in \Fix(T^{2,2,2}_{p,q,r} \circ T^{p,q,r}_{2,2,2})}
    \beta_i(T^{p,q,r}_{2,2,2}(\lambda); \Lambda^{2,2,2}) t^i \z^\lambda
    \\&= \sum_{i \ge 0} \sum_{\mu \in \Lambda^{2,2,2}}
    \beta_i(\mu; \Lambda^{2,2,2}) t^i \z^{T^{2,2,2}_{p,q,r}(\mu)}
    \\&= (T^{2,2,2}_{p,q,r})_*(P^{K[\Lambda^{2,2,2}]}_K(t, \z)). \qedhere
  \end{align*}
\end{proof}

\begin{prp}
  \label{prp : P 2,2,2}
  Let $(p, q, r) = (1, 1, 2), (1, 2, 2)$ or $(2, 2, 2)$.
  Then the multigraded Poincar\'e series of $K[\Lambda^{p,q,r}]$ satisfies
  \begin{align*}
    P^{K[\Lambda^{p,q,r}]}_K(t, \z)
    &= \sum_{k \ge 0} \Bigl( t^k \z^{k c} + t^{k + 1} \z^{a + k c} + t^{k + 1} \z^{b + k c} + t^{k + 2} \z^{a + b + k c} \Bigr)
    \\&= (1 + t \z^a + t \z^b + t^2 \z^{a + b}) \sum_{k \ge 0} t^k \z^{k c}
    \\&= \frac{(1 + t \z^a)(1 + t \z^b)}{1 - t \z^c}.
  \end{align*}
\end{prp}
\begin{proof}
  We first prove the case $(p, q, r) = (1, 2, 2)$ or $(2, 2, 2)$.
  By Theorem~\ref{thm : case 2,2,2}, we have
  \[ \beta_i(\lambda; \Lambda^{p,q,r}) =
    \begin{cases}
      1 & \text{if $(i, \lambda) \in \{\, (k, k c) \mid k \ge 0 \,\}$,} \\
      1 & \text{if $(i, \lambda) \in \{\, (k + 1, a + k c) \mid k \ge 0 \,\}$,} \\
      1 & \text{if $(i, \lambda) \in \{\, (k + 1, b + k c) \mid k \ge 0 \,\}$,} \\
      1 & \text{if $(i, \lambda) \in \{\, (k + 2, a + b + k c) \mid k \ge 0 \,\}$,} \\
      0 & \text{otherwise.}
    \end{cases} \]
  Thus we obtain
  \begin{align*}
    P^{K[\Lambda^{p,q,r}]}_K(t, \z)
    &= \sum_{i \ge 0} \sum_{\lambda \in \Lambda} \beta_i(\lambda; \Lambda^{p,q,r}) t^i \z^\lambda
    \\&= \sum_{k \ge 0} t^k \z^{k c} + \sum_{k \ge 0} t^{k + 1} \z^{a + k c}
    + \sum_{k \ge 0} t^{k + 1} \z^{b + k c} + \sum_{k \ge 0} t^{k + 2} \z^{a + b + k c}
    \\&= \sum_{k \ge 0} \Bigl( t^k \z^{k c} + t^{k + 1} \z^{a + k c}
      + t^{k + 1} \z^{b + k c} + t^{k + 2} \z^{a + b + k c} \Bigr).
  \end{align*}

  We now turn to the case $(p, q, r) = (1, 1, 2)$.
  By Theorem~\ref{thm : case 2,2,2}, we have
  \[ \beta_i(\lambda; \Lambda^{1,1,2}) =
    \begin{cases}
      1 & \text{if $(i, \lambda) = (0, 0)$ or $(1, c)$,} \\
      2 & \text{if $(i, \lambda) \in \{\, (k, k c) \mid k \ge 2 \,\}$,} \\
      1 & \text{if $(i, \lambda) \in \{\, (k + 1, a + k c) \mid k \ge 0 \,\}$,} \\
      1 & \text{if $(i, \lambda) \in \{\, (k + 1, b + k c) \mid k \ge 0 \,\}$,} \\
      0 & \text{otherwise.}
    \end{cases} \]
    Note that $a + b = 2 c$ in $\Lambda^{1,1,2}$.
    Then we obtain
    \begin{align*}
      P^{K[\Lambda^{1,1,2}]}_K(t, \z)
      &= \sum_{i \ge 0} \sum_{\lambda \in \Lambda} \beta_i(\lambda; \Lambda^{1,1,2}) t^i \z^\lambda
      \\&= 1 + t \z^c + 2 \sum_{k \ge 2} t^k \z^{k c}
      + \sum_{k \ge 0} t^{k + 1} \z^{a + k c} + \sum_{k \ge 0} t^{k + 1} \z^{b + k c}
      \\&= \sum_{k \ge 0} t^k \z^{k c} + \sum_{k \ge 0} t^{k + 1} \z^{a + k c}
      + \sum_{k \ge 0} t^{k + 1} \z^{b + k c} + \sum_{k \ge 0} t^{k + 2} \z^{a + b + k c},
      \\&= \sum_{k \ge 0} \Bigl( t^k \z^{k c} + t^{k + 1} \z^{a + k c}
        + t^{k + 1} \z^{b + k c} + t^{k + 2} \z^{a + b + k c} \Bigr). \qedhere
    \end{align*}
\end{proof}

\begin{thm}
  Let $p, q$ and $r$ be positive integers and assume that $r \ge 2$.
  Then the multigraded Poincar\'e series of $K[\Lambda^{p,q,r}]$ satisfies
  \[ P^{K[\Lambda^{p,q,r}]}_K(t, \z) = \frac{(1 + t \z^a)(1 + t \z^b)(1 + t \z^c)}{1 - t^2 \z^{r c}}. \]
\end{thm}
\begin{proof}
  We prove in the case $p, q \ge 2$; the other cases can be proved in a similar way.
  By Proposition~\ref{prp : P p,q,r} and Proposition~\ref{prp : P 2,2,2}, we have
  \begin{align*}
    P^{K[\Lambda^{p,q,r}]}_K(t, \z)
    &= (T^{2,2,2}_{p,q,r})_{*}(P^{K[\Lambda^{2,2,2}]}_K(t, \z))
    \\&= (T^{2,2,2}_{p,q,r})_{*} \Bigl( \sum_{m = 0}^1 \sum_{n = 0}^1 \sum_{k = 0}^\infty
    t^{m + n + k} \z^{m a + n b + k c} \Bigr)
    \\&= (T^{2,2,2}_{p,q,r})_{*} \Bigl( \sum_{m = 0}^1 \sum_{n = 0}^1 \sum_{\ell = 0}^\infty \sum_{k = 0}^1
    t^{m + n + (2 \ell + k)} \z^{m a + n b + (2 \ell + k) c} \Bigr)
    \\&= \sum_{m = 0}^1 \sum_{n = 0}^1 \sum_{\ell = 0}^\infty \sum_{k = 0}^1
    t^{m + n + (2 \ell + k)} \z^{T^{2,2,2}_{p,q,r}(m a + n b + (2 \ell + k) c)}
    \\&= \sum_{m = 0}^1 \sum_{n = 0}^1 \sum_{\ell = 0}^\infty \sum_{k = 0}^1
    t^{m + n + (2 \ell + k)} \z^{m a + n b + (r \ell + k) c}
    \\&= \sum_{m = 0}^1 t^m \z^{m a} \sum_{n = 0}^1 t^n \z^{n b}
    \sum_{\ell = 0}^\infty t^{2 \ell} \z^{r \ell c} \sum_{k = 0}^1 t^k \z^{k c}
    \\&= \frac{(1 + t \z^a)(1 + t \z^b)(1 + t \z^c)}{1 - t^2 \z^{r c}}. \qedhere
 \end{align*}
\end{proof}

\end{document}